\documentclass[a4paper,12pt]{article}

\usepackage[latin1]{inputenc}
\usepackage[T1]{fontenc}
\usepackage{textcomp,graphicx,epsfig}
\usepackage{latexsym,amssymb,amsmath,amsthm}
\usepackage{fancyhdr,rotating,makeidx,eurosym}
\usepackage{lmodern,url,xspace,hyperref,layout}
\usepackage{stmaryrd}

\def\RR{\mathbb{R}}

\def\NN{\mathbb{N}}
\def\ZZ{\mathbb{Z}}
\def\EE{\mathbb{E}}
\def\II{\mathbb{I}}
\def\dd{\mathrm{d}}
\def\ee{\mathrm{e}}
\def\ii{\mathrm{i}}

\newtheorem{defi}{Definition}[section]
\newtheorem{lemm}[defi]{Lemma}
\newtheorem{prop}[defi]{Proposition}

\newtheorem{theo}[defi]{Theorem}

\begin{document}

\title{G\"artner-Ellis condition for squared 
asymptotically stationary Gaussian processes}
\author{Marina Kleptsyna, Alain Le Breton, and Bernard Ycart}
\date{\today}
\maketitle
\begin{abstract}
The G\"artner-Ellis condition for the square of an asymptotically
stationary Gaussian process is established. The same limit holds for the
conditional distribution given any fixed initial point, which entails 
weak multiplicative ergodicity.  
The limit is shown to
be the Laplace transform of a convolution of Gamma distributions with
Poisson compound of exponentials. A
proof based on Wiener-Hopf factorization induces a probabilistic
interpretation of the limit in terms of a regression problem.
\end{abstract}
{\bf MSC 2010:}
%
60G14; 60F10
%
\section{Introduction}
The convergence of the scaled
cumulant generating functions of a sequence of random variables 
implies a large
deviation principle; this is known as the G\"artner-Ellis
condition \cite[p.~43]{DemboZeitouni98}. Our main result establishes
that condition for the square of an asymptotically stationary Gaussian
process.
Reasons for studying squared Gaussian processes
come from different fields: large deviation theory
\cite{Yurinsky95,BrycDembo97}, time series analysis 
\cite{Hannan70}, or
ancestry dependent branching processes \cite{LouhichiYcart15}. 
Since only nonnegative real valued random variables are
considered here, we shall use logarithms of Laplace transforms,
instead of cumulant generating functions.
\begin{theo}
\label{th:main}
Let 
$(X_t)_{t\in\ZZ}$ be a Gaussian process, with
mean $m=(m(t))$ and covariance kernel $K=(K(t,s))$: for all $t,s\in\ZZ$,
$$
\EE[X_t]=m(t)
\quad\mbox{and}\quad
\EE[(X_t-m(t))(X_s-m(s))] = K(t,s)\;.
$$
Assume:
\begin{equation}
\tag{H1}
\sup_{t\in \ZZ} |m(t)| < +\infty\;;
\end{equation}
\begin{equation}
\tag{H2}
\sup_{t\geqslant 1}\,\max_{s=0}^{t-1} 
\sum_{r=0}^{t-1}|K(s,r)|<+\infty\;.
\end{equation}
Assume there exist a constant $m_\infty$ and a positive definite
symmetric function $k$ such that:
\begin{equation}
\tag{H3}
\sum_{t\in\ZZ} |k(t)| <\infty\;;
\end{equation}
\begin{equation}
\tag{H4}
\lim_{t\to+\infty} \frac{1}{t} \sum_{s=0}^{t-1}|m(s)-m_\infty| =
0\;;
\end{equation}
\begin{equation}
\tag{H5}
\lim_{t\to+\infty} \frac{1}{t} 
\sum_{s,r=0}^{t-1} \left|K(s,r)-k(r-s)\right| = 0\;.
\end{equation}
Denote by $f$ the spectral density of $k$:
\begin{equation}
\label{deff}
f(\lambda) = \sum_{t\in\ZZ} \ee^{\ii\lambda t}k(t)\;.
\end{equation} 
For $t\geqslant 0$, consider the following Laplace transform:
\begin{equation}
\label{def:Lt}
L_{t}(\alpha) = 
\EE\left[\exp\left(-\alpha \sum_{s=0}^{t-1} X_s^2\right)\right]\;.
\end{equation}
Then for all $\alpha\geqslant 0$, 
\begin{equation}
\label{limLt}
\lim_{t\to+\infty} \frac{1}{t}\log(L_{t}(\alpha))
= -\ell(\alpha)=-\ell_0(\alpha)-\ell_1(\alpha)
\;,
\end{equation}
with:
\begin{equation}
\label{l0}
\ell_0(\alpha) = 
\frac{1}{4\pi} \int_{0}^{2\pi} \log(1+2\alpha f(\lambda))\,\dd\lambda\;,
\end{equation}
and
\begin{equation}
\label{l1}
\ell_1(\alpha) = 
m_\infty^2\alpha \left(1+2\alpha f(0)\right)^{-1}\;.
\end{equation}
\end{theo}
Theorem \ref{th:main} yields as
a particular case the following result of weak multiplicative ergodicity.
\begin{prop}
\label{prop:main}
Under the hypotheses of Theorem \emph{\ref{th:main}}, consider:
\begin{equation}
\label{def:Lxt}
L_{x,t}(\alpha) = 
\EE_x\left[\exp\left(-\alpha \sum_{s=0}^{t-1} X_s^2\right)\right]\;,
\end{equation}
where $\EE_x$ denotes the conditional expectation given $X_0=x$.\\
Then for all $\alpha\geqslant 0$ and all $x\in \RR$, 
$$
\lim_{t\to+\infty} \frac{1}{t}\log(L_{x,t}(\alpha))
= -\ell(\alpha)
\;,$$
where $\ell$ is defined by \emph{(\ref{limLt})}, \emph{(\ref{l0})}, 
and \emph{(\ref{l1})}.
\end{prop}
The analogue for finite state Markov chains has long been know 
\cite[p.~72]{DemboZeitouni98}. It was extended to  
strong multiplicative ergodicity of exponentially converging Markov
chains by Meyn and his
co-workers: see \cite{KontoyiannisMeyn03}. 
In \cite{Kleptsynaetal14}, the square of a Gauss-Markov
process was studied, strong multiplicative ergodicity was
proved, and the limit was explicitly computed. This motivated the
present  generalization.

The particular case of a centered stationary
process ($m(t)=0,\, K(t,s)=k(t-s)$) can be considered as classical:
in that case, the limit (\ref{l0}) follows from Szeg\H{o}'s theorem on
Toeplitz matrices:
see \cite{GrenanderSzego84}, \cite{BottcherSilbermann99} as a general
reference on Toeplitz matrices, and \cite{Bingham12} for a review
of probabilistic applications of Szeg\H{o}'s theory. The extension to the
centered asymptotically stationary case follows from the notion of
asymptotically equivalent matrices, in the $L^2$ sense: see 
section 7.4 p.~104 of \cite{GrenanderSzego84}, and
\cite{Gray06}. The noncentered stationary case ($m(t)=m_\infty$ and
$K(s,t)=k(s-t)$) has received much less
attention. In Proposition 2.2 of \cite{BrycDembo97}, the Large Deviation
Principle is obtained for a squared noncentered stationary
Gaussian process. There, the centered case is deduced from Szeg\H{o}'s
theorem, while the noncentered case follows from the contraction
principle. 

A different approach to the noncentered stationary case
is proposed here. Instead of the spectral
decomposition and Szeg\H{o}'s theorem, a Wiener-Hopf factorization is
used.
The limits (\ref{l0}) and (\ref{l1}) are both deduced from the asymptotics
of that factorization. 
The technique is close to those developed
in \cite{Kleptsynaetal02}, that were used in
\cite{Kleptsynaetal14}. One advantage is that the coefficients of the Wiener-Hopf
factorization can be given a probabilistic interpretation in terms of a
regression problem. This approach will be detailed in section
\ref{Kleptsyna}.

To go from the stationary to the asymptotically stationary case,
asymptotic equivalence of matrices is needed. But the
classical $L^2$ definition of \cite[section 2.3]{Gray06} does not
suffice for the noncentered 
case. A stronger notion, linked to the $L^1$ norm of vectors instead of the $L^2$
norm, will be developed in section \ref{Szego}. 

Joining the stationary case to asymptotic equivalence, one gets the
conclusion of Theorem \ref{th:main}, but only for small enough values
of $\alpha$. To
deduce that the convergence holds for all $\alpha\geqslant 0$, an extension of
L\'evy's continuity theorem will be used: if both $(L_t(\alpha))^{1/t}$ and
$\ee^{-\ell(\alpha)}$ are Laplace transforms of
probability distributions on $\RR^+$, then the convergence 
over an interval implies weak convergence of measures, hence
the convergence of Laplace transforms for all $\alpha\geqslant 0$. Actually, 
$(L_t(\alpha))^{1/t}$ and $\ee^{-\ell(\alpha)}$ both
are Laplace transforms of infinitely
divisible distributions, more precisely 
convolutions of Gamma distributions with Poisson compounds of
exponentials. Details will be given in section   
\ref{AD}, together with the particular case of a Gauss-Markov process.
\section{The stationary case}
\label{Kleptsyna}
This section treats the stationary case: $m(t)=m_\infty$ and
$K(s,t)=k(t-s)$. We shall denote
by $c_t=(m_\infty)_{s=0,\ldots,t-1}$ the constant vector with
coordinates all equal to $m_\infty$, and
by $H_t$ the Toeplitz matrix with symbol $k$:
$H_t=(k(s-r))_{s,r=0,\ldots,t-1}$.
The main result of this section 
is a particular case of Theorem \ref{th:main}. It
entails Proposition 2.2 of Bryc and Dembo
\cite{BrycDembo97}.
\begin{prop}
\label{stationarycase}
Assume $k$ is a positive definite symmetric function such that
$$
\sum_{t\in\ZZ}|k(t)|=M<+\infty\;,
$$
and denote by $f$ the corresponding spectral density:
$$
f(\lambda) = \sum_{t\in\ZZ} \ee^{\ii\lambda t}k(t)\;.
$$
Let $Z=(Z_t)_{t\in\ZZ}$
be a centered 
stationary  process with covariance function $k$. Let $m_\infty$
be a real. For all $\alpha$ such that $0\leqslant \alpha<1/(2M)$, 
$$
\lim_{t\to +\infty}
\frac{1}{t}\log\left(\EE\left[\exp\left(-\alpha 
\sum_{s=0}^{t-1} (Z_s+m_\infty)^2\right)\right]\right)
=-\ell_0(\alpha)-\ell_1(\alpha)\;,
$$
where $\ell_0(\alpha)$ and $\ell_1(\alpha)$ are defined
by \emph{(\ref{l0})} 
and \emph{(\ref{l1})}.
\end{prop}
Denote by $m_t$ and $K_t$ the mean and covariance matrix of the vector 
$(X_s)_{s=0,\ldots,t-1}$.
The Laplace transform of the squared norm of 
a Gaussian vector has a well known explicit
expression: see for instance \cite[p.~6]{Yurinsky95}. The identity
matrix indexed by $0,\ldots,t-1$ is denoted by $I_t$,
the transpose of a vector $m$ is denoted by $m^*$.
\begin{equation}
\label{explicitLxt}
L_{t}(\alpha) = (\mathrm{det}(I_t+2\alpha K_t))^{-1/2}
\exp(-\alpha m^*_{t}(I_t+2\alpha K_t)^{-1}m_{t})\;,
\end{equation}
In the stationary case, $m_t=c_t$ and $K_t=H_t$.
From (\ref{explicitLxt}), we must prove that the following two limits hold.
\begin{equation}
\label{liml0tstat}
\lim_{t\to +\infty} \frac{1}{2t}\log(\mathrm{det}(I_t+2\alpha H_t)) 
= \ell_0(\alpha)= 
\frac{1}{4\pi} \int_{0}^{2\pi} \log(1+2\alpha
f(\lambda))\,\dd\lambda\;,
\end{equation}
and
\begin{equation}
\label{liml1tstat}
\lim_{t\to +\infty} \frac{\alpha}{t} 
c^*_{t}(I_t+2\alpha H_t)^{-1}c_{t} = 
\ell_1(\alpha)
=
m_\infty^2\alpha(1+2\alpha f(0))^{-1}\;.
\end{equation}
Here, $I_t+2\alpha H_t$ will be interpreted as the covariance matrix
of the random vector $(Y_s)_{s=0,\ldots,t-1}$, from the process
\begin{equation}
\label{modY}
Y=\varepsilon+\sqrt{2\alpha} Z\;,
\end{equation}
where 
$\varepsilon=(\varepsilon_t)_{t\in\ZZ}$ is a sequence of 
i.i.d. standard normal random variables, independent from
$Z$. The limits (\ref{liml0tstat}) and (\ref{liml1tstat}) will be
deduced from a Cholesky decomposition of $I_t+2\alpha H_t$. We begin
with an arbitrary positive definite matrix $A$.
The Cholesky decomposition writes it as the product of a
lower triangular matrix by its transpose. Thus $A^{-1}$ is the
product of an upper triangular matrix by its transpose. Write it
as $A^{-1}=T^*DT$, where $T$ is a unit lower triangular matrix
(diagonal coefficients equal to $1$), and
$D$ is a diagonal matrix with positive coefficients. Denote by $G$ the
lower triangular matrix $DT$. Then $GA=(T^*)^{-1}$ is a unit upper
triangular matrix. Hence the
coefficients $G(s,r)$ of $G$ are uniquely determined by the following
system of linear equations. For $0\leqslant s\leqslant t$,
\begin{equation}
\label{sysGA}
\sum_{r=0}^t G(t,r)\,A(r,s) = \delta_{t,s}\;,
\end{equation}
where $\delta_{t,s}$ denotes the Kronecker symbol: $1$ if $t=s$, $0$
else. Notice that $A^{-1}=G^*D^{-1}G$, and $TAT^{*}=D^{-1}$, where
$D$ is the diagonal matrix with diagonal entries $G(s,s)$. In
particular,
\begin{equation}
\label{detA}
\mathrm{det(A)} = \left(\prod_{s} G(s,s)\right)^{-1}\;, 
\end{equation}
and for any vector $m=(m(r))$,
\begin{equation}
\label{fqA}
m^* A^{-1} m = \sum_s \frac{1}{G(s,s)} 
\left(\sum_{r=0}^s G(s,r)m(r)\right)^{2}\;. 
\end{equation}
Here is the probabilistic interpretation of the coefficients
$G(t,s)$. Consider a centered Gaussian vector $Y$ with covariance
matrix $A$. For $t=0,\ldots,n$, denote by 
$\mathcal{Y}_{\llbracket 0,t\rrbracket }$ the
$\sigma$-algebra generated 
by $Y_0,\ldots,Y_{t}$, and by $\nu_t$ the partial innovation:
$$
\nu_t = Y_t -\EE[Y_t\,|\,\mathcal{Y}_{\llbracket 0,t-1\rrbracket}]\;,
$$
with the convention $\nu_0=Y_0$. Using elementary properties of
Gaussian vectors, it is easy to check that:
\begin{equation}
\label{inno}
\nu_t = \frac{1}{G(t,t)} \sum_{r=0}^{t} G(t,r)\, Y_r\;.
\end{equation}
Moreover, the $\nu_t$'s are independent, and the variance of $\nu_t$
is $1/G(t,t)$. 

When this is applied to
$A=I_t+2\alpha H_t$, another interesting
interpretation arises. For $t=0,\ldots,n$ $(G(t,s))_{s=0,\ldots,t}$ is
the unique solution to the system:
\begin{equation}
\label{systG}
G(t,s) +2\alpha \sum_{r=0}^t G(t,r)\, k(r-s)= \delta_{t,s}\;.
\end{equation}
Observe that the equations (\ref{systG}) are the normal
equations of the regression of the $\varepsilon_t$'s over the $Y_t$'s
in the model (\ref{modY}). Actually, since
$\EE[Y_rY_s]=\delta_{s,r}+2\alpha k(r-s)$ and 
$\EE[\varepsilon_tY_s]=\delta_{t,s}$, setting
\begin{equation}
\label{mut}
\mu_t = G(t,t)\,\nu_t = \sum_{r=0}^t G(t,r)\, Y_r\;,
\end{equation}
equation (\ref{systG}) says that for $s=0,\ldots, t$,
$$
\EE[\mu_t Y_s] = \EE[\varepsilon_t Y_s]\;.
$$
This means that:
$$
\mu_t = \EE[\varepsilon_t \,|\, \mathcal{Y}_{\llbracket 0,t\rrbracket}\,]\;.
$$
Obviously, the $\mu_t$'s are independent, the variance of $\mu_t$ is
$G(t,t)$ and the filtering error is:
$$
\EE[ (\varepsilon_t -\mu_t)^2] = 1-G(t,t)\;.
$$ 
In particular, it follows that $0<G(t,t)<1$.

The asymptotics of $G(t,s)$ will now be related 
to the spectral density $f$.
Denote $g_t(s)=G(t,t-s)$. A change
of index in (\ref{systG}) shows that $(g_t(s))_{s=0,\ldots,t}$ is the
unique solution to the system:
\begin{equation}
\label{systgt}
g_t(s) +2\alpha \sum_{r=0}^t g_t(r)\,k(s-r) = \delta_{s,0}\;.
\end{equation}
\begin{prop}
\label{prop:asymptgt}
Assume $k$ is a positive definite symmetric function such that
$$
\sum_{t\in\ZZ}|k(t)|=M<+\infty\;,
$$
and denote by $f$ the corresponding spectral density:
$$
f(\lambda) = \sum_{t\in\ZZ} \ee^{\ii\lambda t}k(t)\;.
$$
For all $\alpha$ such that $0\leqslant \alpha<1/(2M)$, 
the following equation has a unique solution
in $L^1(\ZZ)$.
\begin{equation}
\label{systginfty}
g(s) +2\alpha \sum_{r=0}^{+\infty} g(r)\,k(s-r) = \delta_{s,0}\;.
\end{equation}
One has:
\begin{equation}
\label{ginfty0}
g(0) = \exp\left(
-\frac{1}{2\pi}\int_0^{2\pi} \log(1+2\alpha f(\lambda))\,\dd\lambda
\right)\;,
\end{equation}
and
\begin{equation}
\label{sumginfty}
\sum_{s=0}^{+\infty} g(s)
=\exp\left(
-\frac{1}{2}\log(1+2\alpha f(0)) -\frac{1}{4\pi}
\int_0^{2\pi}\log(1+2\alpha f(\lambda))\,\dd\lambda
\right)\;.
\end{equation}
Moreover, if $g_t(s)$ is defined for all $0\leqslant s\leqslant t$
by \emph{(\ref{systgt})}, then for all $s\geqslant 0$,
\begin{equation}
\label{limgt}
\lim_{t\to +\infty} g_t(s)= g(s)\;,
\end{equation}
and
\begin{equation}
\label{limsumgt}
\lim_{t\to +\infty} \sum_{s=0}^t g_t(s)= \sum_{s=0}^{+\infty} g(s)\;.
\end{equation}
\end{prop}
The main idea of the proof amounts to writing the Wiener-Hopf
factorization of the operator $I+2\alpha H$. The method is originally 
due to M. G. Krein: 
see section 1.5 of 
\cite{BottcherSilbermann99}, in particular the proof of
Theorem 1.14 p.~17. 
%
\begin{proof}
Conditions of invertibility for Toeplitz operators are well
known. They are treated in section 2.3 and 7.2 of
\cite{BottcherSilbermann99}. Here, the $L^1$ norm of the Toeplitz
operator $H$ with
symbol $k$ is $M$, and the condition $0\leqslant \alpha<1/(2M)$
permits to write the inverse as:
$$
(I+2\alpha H)^{-1} = \sum_{n=0}^{+\infty} (-2\alpha H)^n\;.
$$ 
That the truncated inverse 
$(I_t+2\alpha H_t)^{-1}$ converges to $(I+2\alpha H)^{-1}$ 
is deduced for the $L^2$ case from
\cite[p.~42]{BottcherSilbermann99}. Convergence of entries follows,
hence (\ref{limgt}). To obtain (\ref{limsumgt}), consider 
$\delta_t(s)=g(s)-g_t(s)$. From (\ref{systgt}) and (\ref{systginfty}):
$$
\delta_t(s) = -2\alpha \sum_{r=0}^t k(r-s)\delta_t(r)
-2\alpha \sum_{r=t+1}^{+\infty} g(r) k(r-s)\;.
$$
Hence:
$$
\sum_{s=0}^t | \delta_t(s) |\leqslant
2\alpha \left(\sum_{r=0}^t |\delta_t(r)|\right)
\left(\sum_{s=-\infty}^{+\infty} |k(s)|\right) +2\alpha 
\left(\sum_{s=-\infty}^{+\infty} |k(s)|\right)\left(\sum_{r=t+1}^{+\infty}
  |g(r)|\right)\;.
$$
Thus the following bound is obtained:
$$
\sum_{s=0}^t | \delta_t(s) |\leqslant \frac{2\alpha M}{1-2\alpha M}
\sum_{r=t+1}^{+\infty} |g(r)|\;,
$$
which yields (\ref{limsumgt}). 

The generating function of $(g(s))_{s\geqslant 0}$
will now be related to the spectral density $f$.
Define for all $s\in\ZZ$,
$$
g^+(s) = 
\left\{\begin{array}{ll}g(s)&\mbox{if }s\geqslant 0\\
0&\mbox{else}\end{array}\right.
\quad\mbox{and}\quad
g^-(s) = 
\left\{\begin{array}{ll}g(s)&\mbox{if }s< 0\\
0&\mbox{else.}\end{array}\right.
$$ 
Denote by $F^+$ and $F^-$ the Fourier transforms of $g^+$
and $g^-$:
$$
F^\pm(\lambda) = \sum_{s\in\ZZ} \ee^{\ii s\lambda} g^\pm(s)\;.
$$ 
Take Fourier transforms in both members of (\ref{systginfty}):
$$
F^+(\lambda)+F^-(\lambda)+2\alpha
F^+(\lambda) f(\lambda)  = 1\;.
$$
Or else:
\begin{equation}
\label{eqfourier}
F^+(\lambda)(1+2\alpha f(\lambda))
=1-F^-(\lambda)\;.
\end{equation}
The functions $F^\pm(\lambda)$ can be seen as being defined on the unit
circle. They can be extended into analytic functions, one inside the
unit disk, the other one outside. For $|\zeta|\leqslant 1$:
$$
\bar{F}^+(\zeta) = \sum_{s\geqslant 0} \zeta^s g^+(s)\;,
$$
and for $|\zeta|\geqslant1$:
$$
\bar{F}^-(\zeta) = \sum_{s< 0} \zeta^s g^-(s)\;.
$$
Similarly, we shall denote by $\bar{f}(\zeta)$ the value of
$f(\lambda)$ for $\zeta=\ee^{\ii\lambda}$.
Let $\zeta_0$ be any fixed complex inside the unit disk. In
(\ref{eqfourier}), take logarithms of both members (principal branch), divide by
$2\pi\ii(\zeta-\zeta_0)$, then take the contour integral over the unit
circle.
\begin{equation}
\label{contourfourier}
\begin{array}{cl}
&\displaystyle{\frac{1}{2\pi\ii}\oint_{|\zeta|=1}
\frac{\log(\bar{F}^+(\zeta))}{\zeta-\zeta_0}\,\dd \zeta
+
\frac{1}{2\pi\ii}\oint_{|\zeta|=1}
\frac{\log(1+2\alpha
  \bar{f}(\zeta))}{\zeta-\zeta_0}
\,\dd \zeta}\\[2ex]
=&\displaystyle{
\frac{1}{2\pi\ii}\oint_{|\zeta|=1}
\frac{\log(1-\bar{F}^-(\zeta))}{\zeta-\zeta_0}\,\dd \zeta\;.}
\end{array}
\end{equation}
Since $\bar{F}^+$ is analytic inside the unit disk, the residue of the
first integral is $\bar{F}^+(\zeta_0)$. Let us prove that the integral
in the second member is null. Since the function to be integrated 
is analytic outside the unit circle, the integral has the
same value over any circle with radius $\rho>1$, centered at $0$. 
$$
\frac{1}{2\pi\ii}\oint_{|\zeta|=1}
\frac{\log(1-\bar{F}^-(\zeta))}{\zeta-\zeta_0}\,\dd \zeta
=
\frac{1}{2\pi\ii}\oint_{|\zeta|=\rho>1}
\frac{\log(1-\bar{F}^-(\zeta))}{\zeta-\zeta_0}\,\dd \zeta\;.
$$
As $\rho$ tends to $+\infty$, it can easily be checked that the right
hand side tends to zero.
Thus (\ref{contourfourier}) becomes:
\begin{equation}
\label{logF}
\log(\bar{F}^+(\zeta_0))
=
-\frac{1}{2\pi\ii}\oint_{|\zeta|=1}
\frac{\log(1+2\alpha
  \bar{f}(\zeta))}{\zeta-\zeta_0}
\,\dd \zeta\;.
\end{equation}
Two particular cases are of interest. Consider first $\zeta_0=0$.
\begin{eqnarray*}
\log(\bar{F}^+(0))
&=&
-\frac{1}{2\pi\ii}\oint_{|\zeta|=1}
\frac{\log(1+2\alpha
  \bar{f}(\zeta))}{\zeta}
\,\dd \zeta\\
&=&
-\frac{1}{2\pi}\int_0^{2\pi} \log(1+2\alpha f(\lambda))\,\dd\lambda\;.
\end{eqnarray*}
Since $\bar{F}^+(0)=g(0)$,  (\ref{ginfty0}) follows.

The other particular case is $\zeta_0=1$, but it is on the unit
circle; so a limit has to be taken.
\begin{eqnarray*}
\log(\bar{F}^+(1))
&=&
\lim_{\substack{\zeta_0\to 1\\|\zeta_0|<1}}
\left(-\frac{1}{2\pi\ii}\oint_{|\zeta|=1}
\frac{\log(1+2\alpha
  \bar{f}(1))}{\zeta-\zeta_0}
\,\dd \zeta\right.
\\
&&\hspace*{2cm}\left.
-\frac{1}{2\pi\ii}\oint_{|\zeta|=1}
\frac{\log(1+2\alpha
  \bar{f}(\zeta))-\log(1+2\alpha \bar{f}(1))}{\zeta-\zeta_0}
\,\dd \zeta\right)
\;.
\end{eqnarray*}
The first integral does not depend on $\zeta_0$:
$$
-\frac{1}{2\pi\ii}\oint_{|\zeta|=1}
\frac{\log(1+2\alpha
  \bar{f}(1))}{\zeta-\zeta_0}
\,\dd \zeta
=
-\log(1+2\alpha f(0))\;. 
$$
Since $\bar{f}$ has no singularity at $1$, the limit of the second
integral is:
$$
-\frac{1}{2\pi\ii}\oint_{|\zeta|=1}
\frac{\log(1+2\alpha
  \bar{f}(\zeta))-\log(1+2\alpha \bar{f}(1))}{\zeta-1}
\,\dd \zeta
\;.
$$
Written as a real integral:
$$
-\frac{1}{2\pi}\int_{0}^{2\pi}
\frac{\log(1+2\alpha
  f(\lambda))-\log(1+2\alpha f(0))}{1-\ee^{-\ii\lambda}}
\,\dd \lambda
\;.
$$
For all $\lambda$, the real part of $1/(1-\ee^{-\ii\lambda})$ is
$1/2$. The integral of the imaginary part vanishes, because the
function to be integrated is odd. Finally:
$$
\log(\bar{F}^+(1))=
-\frac{1}{2}\log(1+2\alpha f(0)) -\frac{1}{4\pi}
\int_0^{2\pi}\log(1+2\alpha f(\lambda))\,\dd\lambda\;,
$$
hence (\ref{sumginfty}), since 
$$
\bar{F}^+(1)=F^+(0)=\sum_{s=0}^{+\infty} g(s)\;.
$$
\end{proof}
Here is the probabilistic interpretation. Consider a centered
stationary process $(Y_t)_{t\in\ZZ}$, with covariance function
$A(t,s)=a(t-s)$. For $s\leqslant t$, denote by
$\mathcal{Y}_{\llbracket s,t \rrbracket }$ the
$\sigma$-algebra generated by $(Y_r)_{r=s,\ldots,t}$. Consider again
the partial innovation
$\nu_t=Y_t-\EE[Y_t\,|\,\mathcal{Y}_{\llbracket 0,t-1 \rrbracket}]$. From (\ref{inno}), and
using stationarity, $\nu_t$ has the same distribution as
$$
\eta_t = \frac{1}{G(t,t)} \sum_{r=0}^{t} G(t,t-r)\, Y_{-r}\;,
$$
which is:
$$
\eta_t = Y_0-\EE[Y_0\,|\, \mathcal{Y}_{\llbracket -t,-1 \rrbracket}]\;.
$$
As $t$ tends to infinity, $\eta_t$ converges almost surely to:
$$
\eta_{\infty} = Y_0-\EE[Y_0\,|\,\mathcal{Y}_{\llbracket -\infty,-1 \rrbracket}]\;.
$$
Observe by stationarity that for all $r$,
$$
\eta_{\infty} \stackrel{\mathcal{D}}{=}
Y_r - \EE[Y_r\,|\, \mathcal{Y}_{\llbracket -\infty,r-1 \rrbracket }]\;,
$$
which is the innovation process associated to $Y$. Now the variance of
$\nu_t$, $1/G(t,t)$ tends to the variance of $\eta_\infty$. From the
Szeg\H{o}-Kolmogorov formula (see e.g. 
Theorem 3 p.~137 of \cite{Hannan70}), that variance is:
$$
\exp\left(\frac{1}{2\pi} \int_0^{2\pi} \log(\phi(\lambda))\,\dd
  \lambda \right)\;,
$$
where $\phi(\lambda)$ is the spectral density of $Y$. Let $X$ be a
centered stationary process with covariance function $k$, $\varepsilon$
be a standard Gaussian noise, and $Y=\varepsilon+\sqrt{2\alpha} X$. The spectral
densities $\phi$ of $Y$ and $f$ of $X$ are related by
$\phi(\lambda)=1+2\alpha f(\lambda)$. Hence:
$$
\lim_{t\to+\infty} \mathrm{var}(\nu_t)
=
\lim_{t\to+\infty} \frac{1}{G(t,t)}
=
\exp\left(\frac{1}{2\pi} \int_0^{2\pi} \log(1+2\alpha f(\lambda))\,\dd
  \lambda \right)\;,
$$
which is equivalent to (\ref{ginfty0}).

Alternatively, observe that due to stationarity, $\mu_t$ defined by
(\ref{mut}) has the same distribution as:
$$
\xi_t = \sum_{r=0}^t G(t,t-r)\,Y_{-r}\;,
$$
which is
$$
\xi_t = \EE[ \varepsilon_0\,|\,\mathcal{Y}_{\llbracket -t,0  \rrbracket }]\;.
$$
As $t$ tends to infinity, $\xi_t$ converges a.s. to:
$$
\xi_\infty = \EE[ \varepsilon_0\,|\,\mathcal{Y}_{\llbracket -\infty,0  \rrbracket }]\;.
$$
Of course, since $\EE[\varepsilon_{-s} Y_{-r}] = \delta_{s,r}$, 
for all $s=0,\ldots,t$:
$$
\EE[\xi_t\,\varepsilon_{-s}] = G(t,t-s)\;.
$$
Hence the limiting property (\ref{limgt}) says that: 
$$
\EE[\xi_\infty\,\varepsilon_{-s}]
= \lim_{t\to+\infty} G(t,t-s)=g(s)\;. 
$$
Actually, $\xi_\infty$ admits the representation:
$$
\xi_\infty = \sum_{s=0}^{+\infty} g(s)\, Y_{-s}\;.
$$
Similarly, for all $t$,
$$
\EE[\varepsilon_t\,|\,\mathcal{Y}_{\llbracket -\infty,t  \rrbracket }]
= \sum_{s=0}^{+\infty} g(s)\,Y_{t-s}\;,
$$
which means that $(g(s))$ realizes the optimal causal Wiener filter of
$\varepsilon_t$ from the $Y_{t-s}$'s.
\vskip 2mm\noindent
Now, Proposition \ref{stationarycase} is a straightforward consequence of
Proposition \ref{prop:asymptgt}.
\begin{proof}
Let the coefficients $g_\tau(s)$ be defined by
 (\ref{systgt}). Applying (\ref{detA}) to $A=I_t+2\alpha H_t$, one
 gets:
$$
\left(\mathrm{det}(I_t+2\alpha H_t)\right)^{-1/2} = 
\left(\prod_{\tau=0}^{t-1}g_\tau(0)\right)^{1/2}\;.
$$
Therefore:
$$
\frac{1}{t}\log\left(\mathrm{det}(I_t+2\alpha H_t))^{-1/2}\right)
=\frac{1}{2t} \sum_{\tau=0}^{t-1} \log(g_\tau(0))\;.
$$
From Proposition \ref{prop:asymptgt}:
$$
\lim_{\tau\to +\infty} g_\tau(0) = g(0)=
\exp\left(
-\frac{1}{2\pi}\int_0^{2\pi} \log(1+2\alpha f(\lambda))\,\dd\lambda
\right)\;.
$$
Hence:
$$
\lim_{t\to +\infty}
\frac{1}{t}\log\left(\mathrm{det}(I_t+2\alpha H_t))^{-1/2}\right)
=-\ell_0(\alpha)\;.
$$
Applying now (\ref{fqA}) to $A=I_t+2\alpha H_t$, one gets:
$$
c_t^*G_t^* D_t^{-1} G_tc_t
=m_\infty^2\sum_{\tau=0}^{t-1}
\frac{1}{g_\tau(0)} \left(\sum_{s=0}^\tau g_\tau(s)\right)^2\;.
$$
From Proposition \ref{prop:asymptgt}:
\begin{eqnarray*}
\lim_{\tau\to +\infty} 
\frac{1}{g_\tau(0)} \left(\sum_{s=0}^\tau g_\tau(s)\right)^2
&=&
 \frac{1}{g(0)}\left(\sum_{s=0}^{+\infty} g(s)\right)^2\\
&=&
(1+2\alpha f(0))^{-1}\;.
\end{eqnarray*}
Hence:
$$
\lim_{t\to +\infty} \frac{\alpha}{t} c_t^*(I_t+2\alpha H_t)^{-1}c_t=\ell_1(\alpha)\;.
$$
\end{proof}
\section{Asymptotic equivalence}
\label{Szego}
Proposition \ref{stationarycase} only treats the stationary case. To
extend the result under the hypotheses of Theorem \ref{th:main}, a
notion of asymptotic equivalence of matrices and vectors is
needed. It is developed in this section. 

From (\ref{explicitLxt}), we must prove that under the
 hypotheses of Theorem \ref{th:main}:
\begin{equation}
\label{liml0t}
\lim_{t\to +\infty} \frac{1}{2t}\log(\mathrm{det}(I_t+2\alpha K_t)) 
= \ell_0(\alpha)= 
\frac{1}{4\pi} \int_{0}^{2\pi} \log(1+2\alpha
f(\lambda))\,\dd\lambda\;,
\end{equation}
and
\begin{equation}
\label{liml1t}
\lim_{t\to +\infty} \frac{\alpha}{t} 
m^*_{t}(I_t+2\alpha K_t)^{-1}m_{t} = 
\ell_1(\alpha)
=
m_\infty^2\alpha(1+2\alpha f(0))^{-1}\;.
\end{equation}
If $K_t=H_t$ (centered stationary case), (\ref{liml0t}) is
(\ref{liml0tstat}). It can also be obtained by a
straightforward application of 
Szeg\H{o}'s theorem: see 
\cite{BottcherSilbermann99,Bingham12}. That (\ref{liml0t}) holds
(centered asymptotically stationary case) 
is a consequence of the theory of asymptotically
Toeplitz matrices: see
section 7.4 p.~104 of
\cite{GrenanderSzego84}, and also
\cite[Theorem 4 p. 178]{Gray06}. Asymptotic equivalence
of matrices in Szeg\H{o}'s theory is taken in the $L^2$ sense, which
is weaker than the one considered here. In other words,
(\ref{liml0t}) holds under weaker hypotheses than (H1-5). In order to
prove (\ref{liml1t}), we shall develop 
asymptotic equivalence of
matrices and vectors along the same lines as
\cite[section 2.3]{Gray06}, but in a stronger sense,
replacing $L^2$ by $L^\infty$ and 
$L^1$, for boundedness and
convergence. The norms used here for a vector $v=(v(s))_{s=0,\ldots,t-1}$ are:
$$
\|v\|_\infty = \max_{s=0}^{t-1} |v(s)|
\quad\mbox{and}\quad
\|v\|_1 = \sum_{s=0}^{t-1} |v(s)|\;.
$$
For symmetric matrices, the norm subordinate to
$\|\,\cdot\,\|_\infty$ is equal to the norm subordinate to 
$\|\,\cdot\,\|_1$. It will be denoted by $\|\,\cdot\,\|$ and referred
to as \emph{strong norm}. For
$A=(A(s,r))_{s,r=0,\ldots,t-1}$ such that $A^*=A$,
\begin{eqnarray*}
\|A\|&=&
\max_{s=0}^{t-1}\sum_{r=0}^{t-1}|A(s,r)|
=\max_{\|v\|_\infty=1} \|Av\|_\infty\\
&=&
\max_{r=0}^{t-1}\sum_{s=0}^{t-1}|A(s,r)| 
=\max_{\|v\|_1=1} \|Av\|_1 \;.
\end{eqnarray*}
The following \emph{weak norm} will be denoted by $|A|$:
$$
|A| = \frac{1}{t} \sum_{s,r=0}^{t-1} |A(s,r)|\;.
$$
Clearly, $|A|\leqslant \|A\|$. Moreover, the following bounds hold.
\begin{lemm}
\label{lem:strongweak}
Let $A$ and $B$ be two symmetric matrices. Then:
$$
|AB|\leqslant \|A\|\,|B|
\quad\mbox{and}\quad
|AB|\leqslant |A|\,\|B\|\;.
$$
\end{lemm}
\begin{proof}
$|AB|$ is the arithmetic mean of the $L^1$ norms of column vectors of
$AB$. If $b$ is any column vector of $B$,
$$
\|Ab\|_1 \leqslant \|A\|\,\|b\|_1\;,
$$ 
because the strong norm is subordinate to the $L^1$ norm of
vectors. Hence the first result. For the second result, replace
columns by rows.
\end{proof}
Here is a definition of asymptotic equivalence for vectors.
\begin{defi}
Let $(v_t)_{t\geqslant 0}$ and $(w_t)_{t\geqslant 0}$ 
be two sequences of vectors
such that for all $t\geqslant 0$, $v_t=(v_t(s))_{s=0,\ldots,t-1}$
and $w_t=(w_t(s))_{s=0,\ldots,t-1}$. They are said to be
\emph{asymptotically equivalent} if:
\begin{enumerate}
\item
$\|v_t\|_\infty$ and $\|w_t\|_\infty$ are uniformly bounded,
\item
$\displaystyle{\lim_{t\to +\infty} \frac{1}{t} \|v_t-w_t\|_1 = 0.}$
\end{enumerate} 
Asymptotic equivalence of $(v_t)$ and $(w_t)$ will be denoted by $v_t\sim w_t$.
\end{defi}
\noindent
Hypotheses (H1) and (H4) imply
that $m_t\sim c_t$.
\\\noindent
Asymptotic equivalence for matrices is defined as follows (compare with 
\cite[p.~172]{Gray06}).
\begin{defi}
Let $(A_t)_{t\geqslant 0}$ and $(B_t)_{t\geqslant 0}$ be two sequences
of symmetric matrices, where for all $t\geqslant 0$, 
$A_t=(A_t(s,r))_{s,t=0,\ldots,t-1}$ and
$B_t=(B_t(s,r))_{s,t=0,\ldots,t-1}$. They are said to be
\emph{asymptotically equivalent} if:
\begin{enumerate}
\item
$\|A_t\|$ and $\|B_t\|$ are uniformly bounded,
\item
$\displaystyle{\lim_{t\to +\infty} |A_t-B_t| = 0}$.
\end{enumerate} 
Asymptotic equivalence of $(A_t)$ and $(B_t)$ 
will still be denoted by $A_t\sim B_t$.
\end{defi}
Here are some elementary results, analogous to those stated in 
Theorem 1 p. 172 of \cite{Gray06}.
\begin{lemm}
\label{lem:AEmat}
Let $(A_t), (B_t), (C_t), (D_t)$ be four sequences of symmetric matrices.
\begin{enumerate}
\item
If $A_t\sim B_t$ and $B_t\sim C_t$ then $A_t \sim C_t$.
\item
If $A_t\sim B_t$ and $C_t\sim D_t$ then
$A_t+C_t\sim B_t+D_t$.
\item
If $A_t\sim B_t$ and $C_t\sim D_t$ then
$A_tC_t\sim B_tD_t$.
\item
If $A_t\sim B_t$ and $F$ is an analytic function with radius
$R>\max\|A_t\|, \max\|B_t\|$, then $F(A_t)\sim F(B_t)$.
\end{enumerate}  
\end{lemm}
\begin{proof}
Points 1 and 2 follow from the triangle inequality for the weak
norm. For point 3, because $\|\,\cdot\,\|$ is a norm of matrices,
$\|A_tC_t\|\leqslant \|A_t\|\,\|C_t\|$ and $\|B_tD_t\|\leqslant
\|B_t\|\,\|D_t\|$ are uniformly bounded. Moreover by Lemma
\ref{lem:strongweak},
\begin{eqnarray*}
|A_tC_t - B_tD_t| &\leqslant& |(A_t-B_t)C_t|+|B_t(C_t-D_t)|\\
&\leqslant& |A_t-B_t|\,\|C_t\|+\|B_t\|\,|C_t-D_t|\;.
\end{eqnarray*}
Since $\|C_t\|$ and $\|B_t\|$ are uniformly bounded, and
$$
\lim_{t\to \infty} |A_t-B_t| =
\lim_{t\to \infty} |C_t-D_t|=0\;,
$$
the result follows. For point 4,  
let $F$ be analytic with radius of convergence $R$.
For
$|z|<R$, let
$$
F(z)=\sum_{k=0}^{+\infty} a_k\,z^k\;,
$$  
and
$$
F_n(z)=\sum_{k=0}^{n} a_k\,z^k\;.
$$
The matrices $F(A_t)$, $F(B_t)$ are defined as the limits of
$F_n(A_t)$, $F_n(B_t)$; from the hypothesis, it follows that the
convergence is uniform in $t$. Because $\|\,\cdot \,\|$ is a matrix norm, 
$
\|F(A_t)\| \leqslant F(\|A_t\|)
$
and the same holds for $B_t$: $\|F(A_t)\|$ and $\|F(B_t)\|$ are
uniformly bounded. Let $\epsilon$ be a positive real. Fix $n$ such
that for all $t$,
$$ 
\|F(A_t)-F_n(A_t)\|<\frac{\epsilon}{3}
\quad\mbox{and}\quad
\|F(B_t)-F_n(B_t)\|<\frac{\epsilon}{3}\;.
$$
By induction on $n$ using points 2 and 3, $F_n(A_t)\sim
F_n(B_t)$. There exists $t_0$ such that for all $t>t_0$, 
$$
|F_n(A_t)-F_n(B_t)|<\frac{\epsilon}{3}\;.
$$
Thus for all $t>t_0$,
\begin{eqnarray*}
|F(A_t)-F(B_t)|&\leqslant&
|F(A_t)-F_n(A_t)|+|F_n(A_t)-F_n(B_t)|+|F_n(B_t)-F(B_t)|\\
&\leqslant&
\|F(A_t)-F_n(A_t)\|+|F_n(A_t)-F_n(B_t)|+\|F_n(B_t)-F(B_t)\|\\
&<&\epsilon\;.
\end{eqnarray*}
Hence the result.
\end{proof}
Hypothesis (H3) implies that
$\|H_t\|$ is uniformly bounded, (H2) and (H5) that $K_t\sim
H_t$. Point 4 will be
applied to $F(z)=(1 +2\alpha z)^{-1}$, which has radius of convergence
$R=1/2\alpha$. Let $M$ be defined as:
$$
M = \max\left\{\max_{t\geqslant 1}\|K_t\|\,,\; \sum_{t\in\ZZ} |k(t)|\right\}\;.
$$
For all $\alpha<\alpha_0=1/(2M)$,
\begin{equation}
\label{equivKH}
(I_t+2\alpha K_t)^{-1}
\,\sim\,
(I_t+2\alpha H_t)^{-1}\;.
\end{equation}

Here is the relation between asymptotic equivalence of vectors
and matrices.
\begin{lemm}
\label{lem:AEmatvec}~
\begin{enumerate}
\item
If $A_t\sim B_t$ and $\|v_t\|_\infty$ is uniformly bounded, then
$A_tv_t\sim B_tv_t$.
\item
If $v_t\sim w_t$ and $\|A_t\|$ is uniformly bounded, then
$A_tv_t\sim A_tw_t$.
\end{enumerate}
\end{lemm}
\begin{proof}
That $\|A_tv_t\|_\infty$, $\|B_tv_t\|_\infty$, $\|A_tw_t\|_\infty$ are
uniformly bounded comes from the fact that $\|\,\cdot\,\|$ is
subordinate to $\|\,\cdot\,\|_\infty$. Next for point $1$:
$$
\frac{1}{t} \|(A_t-B_t)v_t\|_1\leqslant \|v_t\|_\infty |A-B|\;.
$$ 
For point 2:
$$
\frac{1}{t} \|A_t(v_t-w_t)\|_1\leqslant 
\frac{1}{t} \|A_t\|\|v_t-w_t\|_1\;.
$$ 
\end{proof}
The relation between asymptotic equivalence of vectors and our goal is
the following.
\begin{lemm}
\label{lem:AEvec}
If $v_t\sim w_t$ and $u_t\sim z_t$, then
$$
\lim_{t\to +\infty} \frac{1}{t}\left(v^*_tu_t-w^*_tz_t\right) = 0\;.
$$
\end{lemm}
\begin{proof}
\begin{eqnarray*}
\frac{1}{t}\left|v^*_tu_t-w^*_tz_t\right|
&\leqslant&
\frac{1}{t}\left(|v^*_t(u_t-z_t)|+|(v^*_t-w^*_t)z_t|\right)\\
&\leqslant&\frac{1}{t}
\left(\|v_t\|_\infty\,\|u_t-z_t\|_1
+
\|z_t\|_\infty\|v_t-w_t\|_1\right)\;.
\end{eqnarray*}
Hence the result.
\end{proof}
Using asymptotic equivalence, (\ref{liml0t}) and (\ref{liml1t}) can
easily be deduced from (\ref{liml0tstat}) and (\ref{liml1tstat}),
 for $0<\alpha<1/(2M)$. We
shall not detail the passage from (\ref{liml0tstat}) to
(\ref{liml0t}): see Theorem 4 p.~178 of \cite{Gray06}. Here is the
passage from (\ref{liml0tstat}) to
(\ref{liml0t}).
For all $\alpha<1/(2M)$, it follows from (\ref{equivKH}) by
point 1 of Lemma \ref{lem:AEmatvec} that
$$
(I_t+2\alpha K_t)^{-1} c_t
\,\sim\,
(I_t+2\alpha H_t)^{-1} c_t\;.
$$
By point 2 of Lemma \ref{lem:AEmatvec}:
$$
(I_t+2\alpha K_t)^{-1} m_{t}
\,\sim\,
(I_t+2\alpha H_t)^{-1} c_t\;.
$$
Lemma
\ref{lem:AEvec} implies:
$$
\lim_{t\to +\infty}\frac{1}{t}
m^*_{t}(I_t+2\alpha K_t) m_{t} = 
\lim_{t\to +\infty}\frac{1}{t}
c_t^*(1+2\alpha H_t)^{-1}c_t\;. 
$$
Hence (\ref{liml1t}).
\vskip 2mm\noindent
Still using asymptotic equivalence, it will now be shown that Proposition
\ref{prop:main} is just a particular case of Theorem \ref{th:main}. Indeed, 
consider the Gaussian process $X^x$ with mean
\begin{equation}
\label{mx}
m_x(t)=\EE[X^x_t]=m(t)+\frac{K(0,t)}{K(0,0)}(x-m(t))\;,
\end{equation}
and covariance function
\begin{equation}
\label{Kbullet}
K^\bullet(t,s)=\EE[(X^x_t-m_x(t))(X^x_s-m_x(s))] = 
K(t,s)-\frac{K(t,0)K(s,0)}{K(0,0)}\;.
\end{equation}
The distribution of $(X^x_t)_{t\in\NN}$ and the 
conditional distribution of $(X_t)_{t\in\NN}$
given $X_0=x$ are the same.
Denote by $m_{x,t}$ and $K^\bullet_t$ the mean and
covariance matrix of $(X^x_s)_{s=0,\ldots,t-1}$. Theorem \ref{th:main}
applies to $X^x$, provided it is proved that $m_{x,t}\sim c_t$ and
$K^\bullet_t\sim H_t$.  By (H1) and
(H2), $\|m_{x,t}\|_\infty$
is uniformly bounded. Moreover from (\ref{mx}),
$$
\frac{1}{t} \|m_{x,t}-m_t\|_1
\leqslant
\frac{|x|+\|m_t\|_\infty}{tK(0,0)}\sum_{s=0}^{t-1}|K(0,s)|
\leqslant
\frac{|x|+\|m_t\|_\infty}{tK(0,0)}\|K_t\|\;,
$$
thus $m_{x,t}\sim m_t$, hence $m_{x,t}\sim c_t$ by transitivity.
Now from (\ref{Kbullet}),
$$
\|K^\bullet_t\|\leqslant \|K_t\|+ 
\max_{r=0}^{t-1}\frac{|K(0,r)|}{K(0,0)}\sum_{s=0}^{t-1}K(0,s)
\leqslant \|K_t\|+
\frac{\|K_t\|^2}{K(0,0)}\;.
$$
Moreover,
$$
|K^\bullet_t-K_t|
\leqslant
\frac{1}{tK(0,0)}\left(\sum_{s=0}^{t-1} |K(0,s)|\right)^2
\leqslant
\frac{\|K_t\|^2}{tK(0,0)}\;,
$$
thus $K^\bullet_t\sim K_t$, hence $K^\bullet_t\sim H_t$ by
transitivity (point 1 of Lemma \ref{lem:AEmat}).
\vskip 2mm\noindent
This section will end with another illustration of asymptotic
equivalence, which is of independent interest and yields an
alternative proof of  (\ref{liml1tstat}).
\begin{prop}
\label{prop:ct}
Let $F$ be analytic with radius of convergence $R>\sum_t|k(t)|$. 
Denote by $d_t^{(\lambda)}$ the vector 
$
d_t^{(\lambda)} = (\ee^{-\ii \lambda s})_{s=0,\ldots,t-1}
$.
 Then:
\begin{equation}
\label{equivHct}
F(H_t) d_t^{(\lambda)}\sim F(f(\lambda)) d_t^{(\lambda)}\;.
\end{equation}
\end{prop}
The function $(\ee^{-\ii \lambda s})_{s\in\ZZ}$ is an eigenfunction of
the Toeplitz operator $H$ with symbol $k$, associated to the eigenvalue
$f(\lambda)$. Thus Proposition \ref{prop:ct} is closely related to
Szeg\H{o}'s theorem: compare with Theorem 5.9
p.~137 of \cite{BottcherSilbermann99}. Notice that $c_t=m_\infty
d_t^{(0)}$: in the particular case
$\lambda=0$ and $F(z)=(1+2\alpha z)^{-1}$, one gets:
$$
(1+2\alpha H_t)^{-1} c_t \sim (1+2\alpha f(0))^{-1} c_t\;,
$$
from which (\ref{liml1tstat}) follows, through Lemma
\ref{lem:AEvec}. If instead of being constant, the asymtotic mean is
periodic, Proposition \ref{prop:ct} still gives an explicit
expression of $\ell_1(\alpha)$. As an example, assume
$m(t)=(-1)^tm_\infty$. Then (\ref{liml1tstat}) holds with: 
$$
\ell_1(\alpha) =
m_\infty^2\alpha(1+2\alpha f(\pi))^{-1}\;. 
$$
\begin{proof}
We first prove (\ref{equivHct}) for $F(z)=z$.
Using the fact that $\|\,\cdot\,\|$ is subordinate to
$\|\,\cdot\,\|_\infty$, 
$$
\|H_t d_t^{(\lambda)}\|_\infty \leqslant \|H_t\|\;.
$$
Therefore $\|H_t d_t^{(\lambda)}\|_\infty$ is uniformly bounded.
For $s=0,\ldots,t-1$, consider the coordinate with index $s$ of
$f(\lambda) d_t^{(\lambda)} - H_td_t^{(\lambda)}$:
\begin{eqnarray*}
\ee^{-i\lambda s}\sum_{r\in\ZZ} k(r)\ee^{\ii\lambda r}
- \sum_{r=0}^{t-1} k(r-s)\ee^{-\ii\lambda r}
&=&
\ee^{-i\lambda s}\sum_{r\in\ZZ} k(r)\ee^{-\ii\lambda r} -
\ee^{-i\lambda s}\sum_{r'=-s}^{t-s-1} k(r')\ee^{-\ii\lambda r'}\\[2ex]
&=&
\ee^{-i\lambda s}
\sum_{r=-\infty}^{-s-1} k(r)\ee^{-\ii\lambda r}
+
\ee^{-\ii\lambda s}\sum_{r=t-s}^{+\infty} k(r)\ee^{-\ii\lambda r}\;.
\end{eqnarray*}
Denote
$$
\delta_-(s)=\sum_{r=-\infty}^{s}k(r)\ee^{-\ii\lambda r}
\quad\mbox{and}\quad
\delta_+(s)=\sum_{r=s}^{+\infty}k(r)\ee^{-\ii\lambda r}\;.
$$
Observe, due to the symmetry of $k$, that $\delta_-(-s)=\overline{\delta_+(s)}$
Thus:
\begin{eqnarray*}
\|f(0) c_t - H_tc_t\|_1
&=&\displaystyle{
\sum_{s=0}^{t-1}
|\delta_-(-s-1)+\delta_+(t-s)|}\\[2ex]
&\leqslant&\displaystyle{
\sum_{s=0}^{t-1}
|\delta_-(-s-1)|
+\sum_{s=0}^{t-1}
|\delta_+(t-s)|
}\\[2ex]
&=&\displaystyle{
2
\sum_{s=1}^{t}
|\delta_+(s)|\;.}
\end{eqnarray*}
The sequence $(|\delta^+(s)|)_{s\in\NN}$ tends to $0$, as a
consequence of the summability of $k$ (H3). Therefore it also tends to
zero in the Ces\`aro sense. Hence the result.

By induction, using the triangle inequality, (\ref{equivHct}) holds for
any polynomial $F_n$.  The rest of the proof is the same as that of
point 4 in Lemma \ref{lem:AEmat}.
\end{proof}

\section{Asymptotic distributions}
\label{AD}
The results of the two previous sections establish that the conclusion
of Theorem \ref{th:main} holds for a small enough $\alpha$. 
To finish the proof, 
the convergence must be extended to all $\alpha\geqslant 0$. 
The following variant of L\'evy's continuity theorem applies
(see Chapter 4 of \cite{Kallenberg97}, in particular Exercise 9 p.~78).
\begin{lemm}
\label{lem:levy}
Let $\pi, \pi_1, \pi_2,\ldots,$ be probability measures on $\RR^+$. Assume
that for some $\alpha_0>0$, and all $\alpha\in[0,\alpha_0[$,
$$
\lim_{n\to\infty} \int_0^{+\infty} \ee^{-\alpha x}\,\dd \pi_n(x)
=
\int_0^{+\infty} \ee^{-\alpha x}\,\dd \pi(x)\;.
$$
Then $(\pi_n)$ converges weakly to $\pi$ and the convergence holds for
all $\alpha\geqslant 0$.
\end{lemm}
To apply this lemma, one has to check that $(L_t(\alpha))^{1/t}$ and 
$\ee^{-\ell(\alpha)}$ 
are the Laplace transforms of probability distributions on
$\RR^+$.
%
%
It turns out that in our case, the function $L_{t}(\alpha)$ 
defined by (\ref{def:Lt})
is the Laplace transform of an infinitely divisible distribution, thus
so are  $(L_{t}(\alpha))^{1/t}$ and its limit.
We give here the probabilistic interpretation of
$\ee^{-\ell_0(\alpha)}$ and 
$\ee^{-\ell_1(\alpha)}$
as the Laplace transforms
of two infinitely divisible distributions. 
Next, the particular case of a Gauss-Markov process
will be considered.

Through an orthogonal transformation diagonalizing its covariance
matrix, the squared norm of any Gaussian vector can be written as the
sum of independent random variables, each being the square of
a Gaussian  variable, 
thus having noncentral chi-squared distribution. If $Z$ is Gaussian
with mean $\mu$ and variance $v$, the Laplace transform of $Z^2$ is:
$$
\phi(\alpha)=(1+2\alpha v)^{-1/2} \exp(-\mu^2\alpha/(1+2\alpha v))\;.
$$
The first factor is the Laplace transform of the Gamma distribution
with shape parameter $1/2$ and scale parameter $2v$. Assuming
$\mu$ and $v$ non null, rewrite the
second factor as:
$$
\exp\left(-\frac{\mu^2}{2v}\left(1-(1+2\alpha v)^{-1}\right)\right)\;.
$$
This is the Laplace transform of a Poisson compound, of the exponential
with expectation $2v$, by the Poisson distribution with rate
$\frac{\mu^2}{2v}$. Therefore, the squared norm of a Gaussian
vector has an infinitely divisible distribution, which is a convolution 
of Gamma distributions with Poisson compounds of exponentials. 
Squared Gaussian vectors 
have received a lot of attention, since even in
dimension $2$, the mean and covariance matrix must satisfy certain
conditions for the distribution of the vector to be infinitely
divisible \cite{MarcusRosen09}. 
Yet the sum of coordinates of such a vector always has an
infinitely divisible distribution. 

For all $t$, the distribution with Laplace transform 
$(L_{t}(\alpha))^{1/t}$ is the convolution of Gamma distributions with
Poisson compounds of exponentials. As $t$ tends to infinity,
$(L_{t}(\alpha))^{1/t}$ tends to $\ee^{-\ell_0(\alpha)}\,\ee^{-\ell_1(\alpha)}$.
The first factor $\ee^{-\ell_0(\alpha)}$ is the Laplace transform of a
limit of convolutions
of Gamma distributions, which belongs to the Thorin class 
$T(\RR^+)$ (see \cite{Bondesson92} as a general
reference). 
Consider now $\ee^{-\ell_1(\alpha)}$. Rewrite $\ell_1(\alpha)$ as:
\begin{eqnarray*}
\ell_1(\alpha) &=& 
m_\infty^2\alpha \left(1+2 \alpha f(0)\right)^{-1}\\
&=&
\frac{m_\infty^2}{2 f(0)}\left(1-(1+2\alpha f(0))^{-1}\right)\;.
\end{eqnarray*}
Thus $\ee^{-\ell_1(\alpha)}$ is the Laplace transform of a Poisson compound, of the
exponential distribution with expectation $2 f(0)$, by the
Poisson distribution with parameter $\frac{m_\infty^2}{2 f(0)}$.

As an illustrating example, consider the Gauss-Markov process defined as follows.
Let $\theta$ be a real such that $-1<\theta<1$. Let
$(\varepsilon_t)_{t\geqslant 1}$ be a sequence of i.i.d. standard Gaussian
random variables. Let $Y_0$, independent from the sequence
$(\varepsilon_t)_{t\geqslant 1}$, follow the normal
$\mathcal{N}(0,(1-\theta^2)^{-1})$ distribution. For all $t\geqslant
1$ let:
$$
Y_t = \theta Y_{t-1}+\varepsilon_t\;. 
$$
Thus $(Y_t)_{t\in\NN}$ 
is a stationary centered auto-regressive process.  
Consider the noncentered process
$(X_t)_{t\in\NN}$, with $X_t=Y_t+m_\infty$. 
This is the case considered in \cite{Kleptsynaetal14}, where a
stronger result was proved. Formula (10) p.~72 of that reference
matches (\ref{l0}) and (\ref{l1}) here. Indeed, the spectral density is:
$$
f(\lambda) = \frac{1}{1+\theta^2-2\theta \cos(\lambda)}\;.
$$
Write $\ell_0(\alpha)$  as a contour integral over the unit
circle.
$$
\ell_0(\alpha)=\frac{1}{4\pi}\int_0^{2\pi} \log(1+2\alpha
f(\lambda))\,\dd\lambda
=
\frac{1}{4\pi \ii}\oint_{|\zeta|=1} \frac{1}{\zeta}
\log\left(1+\frac{2\alpha}{1+\theta^2-\theta(\frac{1}{\zeta}+\zeta)}\right)
\,\dd\zeta\;.
$$ 
Now:
$$
1+\frac{2\alpha}{1+\theta^2-\theta(\frac{1}{\zeta}+\zeta)}
=\frac{\zeta^2-(\theta+\frac{1}{\theta}+\frac{2\alpha}{\theta})\zeta+1}
{\zeta^2-(\theta+\frac{1}{\theta})\zeta+1}\;.
$$
Observe that the two roots of the numerator have the same sign as $\theta$, and
their product is
$1$. Denote them by $\zeta^-$ and $\zeta^+$, so that
$0<|\zeta^-|<1<|\zeta^+|$. The two roots of the denominator are $\theta$
and $\frac{1}{\theta}$. The function to be integrated has 5 poles,
among which $0, \theta, \zeta^-$ are inside the unit disk, $\frac{1}{\theta},
\zeta^+$ are outside. Rewrite $\ell_0$ as:
$$
\ell_0(\alpha)=
\frac{1}{4\pi \ii}\oint_{|\zeta|=1} \frac{1}{\zeta}
\log\left(\frac{\zeta-\zeta^{-}}{\zeta-\theta}\right)
\,\dd\zeta
+
\frac{1}{4\pi \ii}\oint_{|\zeta|=1} \frac{1}{\zeta}
\log\left(\frac{\zeta-\zeta^{+}}{\zeta-\frac{1}{\theta}}\right)
\,\dd\zeta\;.
$$ 
The first integral is null, since
$$
\oint_{|\zeta|=1} \frac{1}{\zeta}
\log\left(\zeta-\zeta^{-}\right)
\,\dd\zeta = 
\oint_{|\zeta|=1} \frac{1}{\zeta}
\log\left(\zeta-\theta\right)
\,\dd\zeta\;,
$$
the two functions having the same residues inside the unit disk.
The second integral is:
$$
\frac{1}{4\pi \ii}\oint_{|\zeta|=1} \frac{1}{\zeta}
\log\left(\frac{\zeta-\zeta^{+}}{\zeta-\frac{1}{\theta}}\right)
\,\dd\zeta
=\frac{1}{2}\log(\theta \zeta^+)\;.
$$
Therefore:
$$
\ell_0(\alpha)= \frac{1}{2}\log(\theta \zeta^+)
=\frac{1}{2}\log\left(\frac{1}{2}\left(
\theta^2+1+2\alpha
+\sqrt{((\theta+1)^2+2\alpha)((\theta-1)^2+2\alpha)}
\right)\right)
\;.
$$
The expression of $\ell_1$ is:
$$
\ell_1(\alpha) = \frac{m_\infty^2\alpha
  (1-\theta)^2}{(1-\theta)^2+2\alpha}\;.
$$
It turns out that the probability distribution with Laplace transform
$$
\ee^{-\ell_0(\alpha)} =
\left(\frac{1}{2}\left(
\theta^2+1+2\alpha
+\sqrt{((\theta+1)^2+2\alpha)((\theta-1)^2+2\alpha)}
\right)\right)^{-1/2}\;, 
$$
has an explicit density $f_0(x)$ defined on $(0,+\infty)$, which is
related to the modified Bessel function of the first kind, with order $1/2$ 
(compare with formula (3.10) p.~437 in \cite{Feller71}).
\begin{eqnarray*}
f_0(x) &=& 
\ee^{-\frac{1+\theta^2}{2}x} \left(2^{-1}|\theta|^{-1/2} x^{-1}
I_{1/2}(|\theta|x)\right)\\
&=&
\ee^{-\frac{1+\theta^2}{2}x} \left((2\pi)^{-1/2}|\theta|^{-1} x^{-3/2}
\sinh(|\theta|x)\right)\;.
\end{eqnarray*}
%

\end{document}